\documentclass[11pt]{article}

\usepackage{epsf,amsmath,amssymb,amsfonts,amsthm}
\usepackage{graphicx, color}
\usepackage{listings}

\newtheorem*{theorem}{Theorem}

\newtheorem*{lemma}{Lemma}

\begin{document}

\title{Not all partial cubes are $\Theta$-graceful}

\date{}
%\date{\today}

\author{
 Nathann Cohen\\
 CNRS, LRI, Univ. Paris Sud\\
Orsay, France\\
  \texttt{nathann.cohen@gmail.com}
\and
   Matja\v{z} Kov\v{s}e\\
   School of Basic Sciences, 
IIT Bhubaneswar,\\
Bhubaneswar, 
India\\
  \texttt{matjaz.kovse@gmail.com }
}

\maketitle

\begin{abstract}
It is shown that the graph obtained by merging two vertices of two 4-cycles is not a $\Theta$-graceful partial cube, thus answering in the negative a question by Bre\v{s}ar and Klav\v{z}ar from \cite{brkl-06}, who asked whether every partial cube is $\Theta$-graceful.
\end{abstract}
\bigskip\noindent


\bigskip\noindent
{\bf MR Subject Classifications: 05C78, 05C12}
\bigskip\noindent

A graph $G$ with $m$ edges is called {\em graceful} if there exists an injection $f:V(G) \rightarrow \{0,1,\ldots,m\}$ such that the edge labels, defined by $|f(x)-f(y)|$ for an edge $xy$, are pairwise distinct. The famous Ringel-Kotzig conjecture says that all trees are graceful, see a dynamic survey \cite{gal} for known classes of trees and other graphs which are graceful, and for the state of the art of graph labelings, an area started by the seminal paper of Rosa \cite{ro-67}.

The vertices of the {\em $d$-dimensional hypercube} are formed by all binary tuples of length $d$, two vertices being adjacent if the corresponding tuples differ in exactly one coordinate.  A subgraph $H$ of a graph $G$ is called {\em isometric} if the geodetic distance $d_H(u,v)$ in $H$ between any two vertices $u,v$ of $H$ is equal to their distance $d_G(u,v)$ in $G$. Isometric subgraphs of hypercubes are called {\em partial cubes} (cf. \cite{brkl-06, djok-73, klli-03, klmu-99, ov-11, wink-84}). The smallest dimension of a hypercube containing an isometric embedding of $G$ is called the {\em  isometric dimension} of $G$. A {\em median graph} is a graph in which every three vertices $u,v$ and $w$ have a unique median: a vertex $m$ that belongs to shortest paths between each pair of $u,v$ and $w$. Every tree is a median graph. Median graphs represent one of the most studied classes of partial cubes, see \cite{klmu-99}.

Two edges $e=xy$ and $f=uv$ of $G$ are in the Djokovi\'c-Winkler 
\cite{djok-73,wink-84} relation $\Theta$ if $$d(x,u) + d(y,v) \not= d(x,v) + d(y,u).$$
The relation $\Theta$ is reflexive and symmetric. Winkler~\cite{wink-84} proved that a connected bipartite graph is a partial cube if and only if $\Theta$ is a transitive relation. Therefore, $\Theta$ is an equivalence relation on a partial cube $G$ and so partitions the edge set of $G$ into the so-called $\Theta${\em-classes}. The number of $\Theta$-classes of a partial cube $G$ is equal to the isometric dimension of $G$.
An {\em isometric cover} $G_1,G_2$ of a connected graph $G$ consists of two isometric subgraphs $G_1$ and $G_2$ of $G$ such that $G=G_1\cup G_2$ and  $G_1 \cap G_2 \not=\emptyset$. Let $\widetilde{G}_1$ and $\widetilde{G}_2$ be isomorphic copies of $G_1$ and $G_2$, respectively. For any vertex $u \in G_i$, $ i \in \{1, 2\}$, let $\widetilde{u}_i$ be the corresponding vertex in $\widetilde{G}_i$. Then the  {\em expansion of $G$ with respect to} $G_1, G_2$ is the graph $\widetilde{G}$ obtained from the disjoint union of $\widetilde{G}_1$ and $\widetilde{G}_2$, where for any $u \in G_{1} \cap G_{2}$ the vertices $\widetilde{u}_1$ and $\widetilde{u}_2$ are joined by an edge. %A {\em contraction} is the reverse operation to the expansion.
Chepoi~\cite{chepoi1988} proved that a graph is a partial cube if and only if it can be obtained from $K_1$ by a sequence of expansions. Chepoi followed the approach of Mulder~\cite{mulder1978,mulder1980} who previously proved an analogous result for median graphs. 

Bre\v{s}ar and Klav\v{z}ar~\cite{brkl-06} introduced a new kind of labeling of partial cubes, called 
$\Theta$-graceful labeling. For a partial cube $G$, on $n$ vertices, a bijection $f:V(G) \rightarrow \{0,1,\ldots,n-1\}$, is called {\em $\Theta$-graceful labeling of} $G$ if all edges in each $\Theta$-class of $G$ receive the same label, and distinct $\Theta$-classes get distinct labels, where the labeling of the edges is defined by $|f(x)-f(y)|$ for every edge $xy$. When such a labeling exists $G$ is called a {\em $\Theta$-graceful partial cube}. Trees are partial cubes, with every $\Theta$-class of a tree consisting of a single edge. Therefore $\Theta$-graceful labelings coincide with graceful labelings on trees. It has been shown that hypercubes, even cycles, Fibonacci cubes and Lexicographic subcubes are $\Theta$-graceful. In \cite{brkl-06} a question has been proposed on whether every partial cube is $\Theta$-graceful. Note that a positive answer to the question would provide a positive solution of the Ringel-Kotzig conjecture.

A $\Theta$-graceful labeling $f$ is called a {\em consistent $\Theta$-graceful labeling} if for every two edges $xy$ and $uv$ of $G$ with $xy\Theta uv$ the following holds:
\begin{equation}
  \label{eqn:consistency}
  d(x,u) < d(x,v) \ \Rightarrow \ f(x) + f(v) = f(y) + f(u)
\end{equation}

%In \cite{brkl-06} only one example of a non consistent $\Theta$-graceful  labeling of a partial cube is presented, that is a $\Theta$-graceful labeling of a cycle on 16 vertices.

\begin{lemma}
\label{lemma1}
Let $G$ be a partial cube with a $\Theta$-graceful labeling $f$. For every isometric 4-cycle $C$ of $G$, equation~(\ref{eqn:consistency}) holds for $C$ labeled with $f$.
\end{lemma}

\begin{proof}
Let $\{v_1,\ldots, v_4\}$ denote the vertices of an isometric 4-cycle of $G$, $f$ a $\Theta$-graceful labeling of $G$, and $f_i = f(v_i)$, for $1\geq i \geq 4$.

Let $x = |f_1- f_2| = |f_4- f_3|$, $y = |f_2- f_3| = |f_4- f_1|$. A consistent $\Theta$-graceful labeling of $C$ must satisfy the following conditions:
\begin{description}
\item[(A)] $f_1- f_2 = f_4- f_3$,
\item[(B)] $f_2- f_3 = -(f_4- f_1)$.
\end{description}

Let $W(f)= (f_1-f_2, f_2-f_3, f_3-f_4, f_4-f_1, )$ denote the difference vector of edges of a 4-cycle.
For example one of the possible difference vectors of a consistent $\Theta$-graceful labeling $f$ might be: $(+x,-y,-x,+y)$. Note that the sum of the values of $W(f)$ is always zero.

Suppose on the contrary that at least one of the conditions (A) and (B) is not satisfied. \\

{\bf Case 1.} Exactly one of the conditions is not satisfied, W. l. o. g.  let it be condition (A). Therefore e. g. $W(f) = (+x, +y, +x, -y)$ for two non-zero relative integers $x,y\in\mathbb Z$. The sum of the values of $W(f)$ is consequently equal to $2x$ or to $-2x$, a non-zero value, which gives a contradiction.\\

{\bf Case 2.} Both conditions are not satisfied. Therefore e. g. 
$W(f) = (+x, +y, +x, +y)$. The sum of $W(f)$ being zero, we obtain the equality $x=-y$, which cannot happen in a $\Theta$-graceful labeling and so this is another contradiction.
\end{proof}

Moreover, the consistency condition (\ref{eqn:consistency}) in a $\Theta$-graceful labeling of an isometric 4-cycle with vertices $\{v_1,\dots,v_4\}$ implies: $f_1+f_3=f_2+f_4$, hence the sum of the two labels assigned to an antipodal pair of vertices is constant. Therefore the smallest and largest value of the labeling are assigned to an antipodal pair of vertices.

Let $G = (V, A)$ be an oriented graph with the set of vertices $V$ and the set of arcs $A$. An arc $(x, y) \in A$ is considered to be directed from $x$ to $y$, and $y$ is called the head and $x$ is called the tail of the arc. The {\em indegree} $deg^+(v)$ of a vertex $v \in G$ is the number of arcs with head $v$, and the {\em outdegree} $deg^-(v)$ of $v \in G$  is the number of arcs with tail $v$. A $\Theta$-graceful labeling of a partial cube $G$ defines an orientation on $G$: each arc is directed from larger vertex  to lower vertex label. Hence for a partial cube $G$ with $\Theta$-graceful labeling, the maximum label must be assigned to a vertex with $deg^+(v)=0$, while the minimum label must be assigned to a vertex with $deg^-(v)=0$. 

\noindent Let $G_8$ denote the graph obtained by merging two vertices of two 4-cycles, see Figure \ref{fig1}.
\vspace{5mm}

\begin{figure}[h!]
\begin{center}
\includegraphics[width=0.20\textwidth]{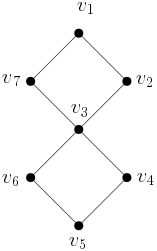}
\caption{Graph $G_8$.}
\label{fig1}
\end{center}
\end{figure}

Although the following result can be easily checked by computer, we provide complete proof as a general approach for other possible generalizations of graceful-like labelings of graphs.

\begin{theorem}
\label{thm.G8}
$G_8$ is not a $\Theta$-graceful partial cube.
\end{theorem}

\begin{proof}
Let $v_1,\ldots, v_{7}$ denote the vertices of $G_8$ as depicted in the Figure \ref{fig1}.  Suppose that $G_8$ admits a $\Theta$-graceful labeling $f$. Let $f_i = f(v_i)$, for $i  \in \{1, \ldots, 7\}$.  

Note that there are four different $\Theta$-classes: $\{v_1v_2, v_3v_7\}$, $\{v_1v_7,v_2v_3\}$, $\{v_3v_4, v_5v_{6}\}$ and $\{v_3v_6, v_4v_5\}$. Hence $|f_1- f_2| = |f_3 - f_7| = a$, $|f_1- f_7| = |f_2 -f_3| = b$, $|f_3- f_4 | = |f_5 - f_6| = c$ and $| f_3 - f_{6}|  = |f_4 - f_5| =d$. Let $\sigma^1= f_1 + f_3$ , $\sigma^2= f_3 + f_5$ - the constant sum of the pairs of labelings assigned to antipodal pairs in cycles. Note that $3\leq \sigma_1, \sigma_2 \leq 9$.

There are three distinct possibilities for choosing a vertex with the maximum label 6: $v_3$, one of the neighbours of $v_3$ - w. l. o. g. let this be $v_2$, one two non neighbours of $v_3$ - w. l. o. g. let this be $v_1$. Altogether inducing seven different orientations of $G_8$ as depicted on Figure \ref{fig2}.

\begin{figure}[h!]
\begin{center}
\includegraphics[width=1.00\textwidth]{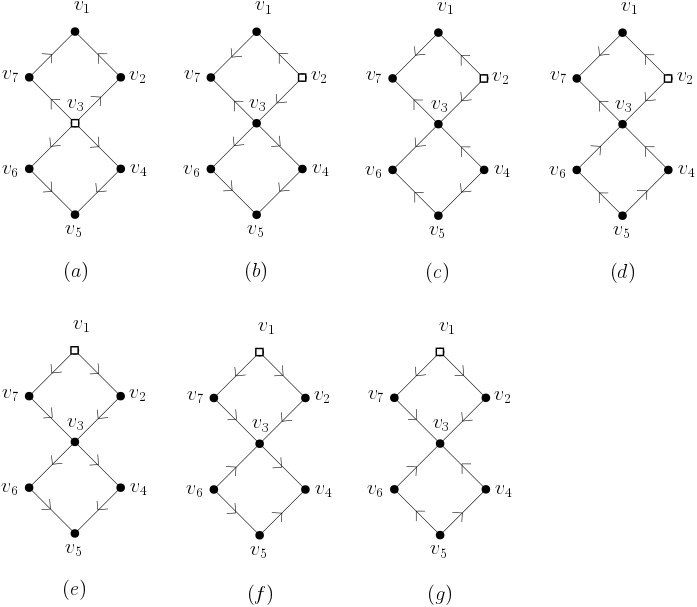}
\caption{Seven separate cases: (a) $f_3=6$, (b) $f_2=6$, $deg^+(v_3)=1$, (c) $f_2=6$, $deg^+(v_3)=2$, (d) $f_2=6$, $deg^+(v_3)=3$, (e) $f_1=6$, $deg^+(v_3)=2$, (f) $f_1=6$, $deg^+(v_3)=3$, (g) $f_1=6$, $deg^+(v_3)=4$,
and their corresponding orientations.}
\label{fig2}
\end{center}
\end{figure}

{\bf Case (a)} $f_3=6$.\\
Then either $f_1=0$ or $f_5=0$. W. l. o. g. let $f_1=0$. Hence $\sigma^1=6$ and we have two subcases.\\

{\bf Subcase (a.1)} $\{f_2, f_7\}= \{1,5\}$.\\
Hence $\{f_4, f_5, f_6\}= \{2,3,4\}$ and therefore $f_3 +f_5 \in \{8,9,10\}$. Hence $f_3 + f_5 > f_4 +f_6$, which is a contradiction.\\

{\bf Subcase (a.2)} $\{f_2, f_7\}= \{2,4\}$.\\
Hence $\{f_4, f_5, f_6\}= \{1,3,5\}$ and therefore $f_3 +f_5 \in \{7,9,11\}$. Hence $f_3 + f_5 < f_4 +f_6$ or $f_3 + f_5 > f_4 +f_6$, which is a contradiction.\\

%%%%%%%%%%%%%%%%%%%%%%%%%%%%%%%%%
{\bf Case (b)} $f_2=6$ and $deg^+(v_3)=1$.\\
%%%%%%%%%%%%%%%%%%%%%%%%%%%%%%%%%
Observing the orientation of $G_8$ as depicted in the case (b) in the Figure \ref{fig2}, it follows that $f_3 > f_4, f_5, f_6, f_7$ and moreover that $f_3 \in \{4,5\}$.\\

{\bf Subcase (b.1)} $f_3=4$.\\
Hence $\{f_4, f_5, f_6, f_7\}= \{0,1,2,3\}$. Therefore $f_1=5$ and $f_1+f_3=\sigma^1=9$. Which further implies $f_7=3$. Hence $v_5=0$ and 
$f_4,f_6 \in \{1,2\}$. Hence $f_3+f_5 > f_4 +f_6$, which is a contradiction.\\

{\bf Subcase (b.2)} $f_3=5$.\\

{\bf Subcase (b.2.1)} $f_5=0$.\\
Since $b=1$, it follows that $\{f_1,f_7\}=\{1,2\}$ or $\{f_1,f_7\}=\{3,4\}$. Since $\sigma_2=5$ it follows that $\{f_4,f_6\}=\{2,3\}$ or $\{f_4,f_6\}=\{1,4\}$, which is in contradiction with both possible choices for $f_1$ and $f_7$.\\

{\bf Subcase (b.2.2)} $f_7=0$.\\
Hence $\sigma_1=6$ and therefore $f_1=1$. Hence $f_5=1$ and $\{f_4, f_6\}= \{2,3\}$. Hence $f_3+f_5 > f_4 +f_6$, which is a contradiction.\\

%%%%%%%%%%%%%%%%%%%%%%%%%%%%%%%%%
{\bf Case (c)} $f_2=6$ and $deg^+(v_3)=2$.\\
%%%%%%%%%%%%%%%%%%%%%%%%%%%%%%%%%
Observing the orientation of $G_8$ as depicted in the case (c) in the Figure \ref{fig2}, it follows that there are only two vertices of outdegree 0, hence $f_7=0$ or $f_6=0$. Since $deg^+(v_3)=deg^+(v_6)=2$ and $deg^+(v_5)=1$ it follows that $f_3,f_5,f_6 \neq 5$, hence  $f_1=5$ or $f_4=5$.\\

{\bf Subcase (c.1)} $f_7=0$.\\
Since $deg^-(v_3)=2$ and $f_4 > f_5 > f_6$ it follows that $f_3, f_4, f_5 \neq 1$. Together with $\sigma_1=6$ and $f_3 \neq 5$ it follows that $f_6=1$. Moreover this implies that $f_1 \neq5$, hence $f_4=5$. Hence $\sigma_2=6$, which is in contradiction with $\sigma_1=6$.\\

{\bf Subcase (c.2)} $f_6=0$.\\
Since $f_1 > f_7$, $deg^-(v_3)=2$, and $deg^-(v_4)=2$, it follows that $f_1, f_3, f_4 \neq 1$. Hence $f_5=1$ or $f_7=1$.\\

{\bf Subcase (c.2.1)} $f_5=1$.\\
Hence $c=1$ and therefore $f_1 \neq 5$. Hence $f_4=5$. Therefore $\sigma_2=5$,  $f_3=4$. It also implies $b=2$, which is in contradiction with $\{f_1,f_7\}= \{2,3\}$ and $|f_1 - f_7|=1$.\\

{\bf Subcase (c.2.2)} $f_7=1$.\\
Hence $\sigma_1=7$. It follows that $\sigma_2 \neq 3,4$. Hence $\sigma_2=5$ and therefore $f_4=5$ and $\{f_3,f_5\}=\{2,3\}$. Therefore $f_1=4$. Hence $f_3=3$ and finally $f_5=2$. Therefore $a=c=2$ and $b=d=3$, which is a contradiction.\\

%%%%%%%%%%%%%%%%%%%%%%%%%%%%%%%%%
{\bf Case (d)} $f_2=6$ and $deg^+(v_3)=3$.\\
%%%%%%%%%%%%%%%%%%%%%%%%%%%%%%%%%
Observing the orientation of $G_8$ as depicted in the case (d) in the Figure \ref{fig2}, it follows that $v_7$ is the only vertex of outdegree 0, hence $f_7=0$ and $\sigma_1=6$. Since $f_5> f_4 > f_3 > v_7$, and $f_6 > f_3 > v_7$ it follows that $\{ f_1, f_3\}=\{1,5 \}$. Since $deg^+(v_3)=3$ it follows that $f_3 \neq 5$. Hence $f_1=5$, $f_3=1$, $a=1$ and $b=4$. Hence $\sigma_2=5$, $f_5=4$ and $\{f_4, f_6 \}=\{2,3\}$. Therefore $\{c,d\}=\{1,2\}$, which is in contradiction with $a=1$.\\

%%%%%%%%%%%%%%%%%%%%%%%%%%%%%%%%%
{\bf Case (e)} $f_1=6$ and $deg^+(v_3)=2$.\\
%%%%%%%%%%%%%%%%%%%%%%%%%%%%%%%%%
Observing the orientation of $G_8$ as depicted in the case (e) in the Figure \ref{fig2}, it follows that $f_3=3, f_5=0$ and $\{f_2,f_7\}= \{4,5\}$ and $\{f_4,f_6\}= \{1,2\}$. Moreover $\{a,b\}= \{1,2\}$ and $\{c,d\}= \{1,2\}$, which is a contradiction.\\

%%%%%%%%%%%%%%%%%%%%%%%%%%%%%%%%%
{\bf Case (f)} $f_1=6$ and $deg^+(v_3)=3$.\\
%%%%%%%%%%%%%%%%%%%%%%%%%%%%%%%%%
Observing the orientation of $G_8$ as depicted in the case (f) in the Figure \ref{fig2}, it follows that $v_4$ is the only vertex of outdegree 0, hence $f_4=0$. Moreover $f_2,f_6,f_7 \neq 1$.\\

{\bf Subcase (f.1)} $f_3=1$.\\
Hence $\sigma_1=7$ and $c=1$. Hence $f_2,f_7 \neq 5$ and therefore $f_6=5$ and $\sigma_2=5$. Hence $f_5=4$ and $d=4$. Therefore $\{f_2,f_7\}=\{2,3\}$. Hence  $\{a,b\}=\{3,4\}$, which is in contradiction with $d=4$.\\

{\bf Subcase (f.2)} $f_5=1$.\\
Hence $d=1$. Since $f_2, f_7 > f_3 > f_4$ and $f_6 > f_3, f_5$ it follows that $f_3=2$. Therefore $\sigma_1=8$, hence $\{f_2,f_7\}=\{3,5\}$. Moreover $\sigma_2=3$ and therefore $f_6=3$, which is in contradiction with $\{f_2,f_7\}=\{3,5\}$.\\

%%%%%%%%%%%%%%%%%%%%%%%%%%%%%%%%%
{\bf Case (g)} $f_1=6$ and $deg^+(v_3)=4$.\\
%%%%%%%%%%%%%%%%%%%%%%%%%%%%%%%%%
Observing the orientation of $G_8$ as depicted in the case (g) in the Figure \ref{fig2}, it follows that there is no vertex of outdegree 0, hence no vertex can be labeled with 0, a contradiction.
\end{proof}

Alternatively, one can check the claim from the previous theorem by computer with a short Python code (cf Fig~\ref{code}).

\begin{figure}[h!]
\begin{lstlisting}[language=Python]
from itertools import permutations
for f1,f2,f3,f4,f5,f6,f7 in permutations(range(7)):
  if (abs(f1-f7) == abs(f2-f3) and abs(f1-f2) == abs(f7-f3) and
      abs(f3-f6) == abs(f5-f4) and abs(f6-f5) == abs(f3-f4) and
      len({abs(f1-f7),abs(f1-f2),abs(f3-f4),abs(f4-f5)]}) == 4):
    print("One solution found")
\end{lstlisting}
\caption{Exploring all labelings of $G_8$}
\label{code}
\end{figure}

Let $S(G)$ denote the graph obtained from a graph $G$ by subdividing once each of its edges. In \cite{klli-03} it has been shown that $S(K_n)$ is a partial cube, where $K_n$ denotes the complete graph on $n$ vertices (see Fig.~\ref{fig3}). Let $Q_3^-$ denote graph obtained by deleting a vertex in 3-dimensional hypercube. Let $C(Q_3^-)$ denote graph obtained by expanding the four vertices inducing a claw in  $Q_3^-$ (see Fig.~\ref{fig3}).

\begin{figure}[htbp]
\begin{center}
\includegraphics[width=\textwidth]{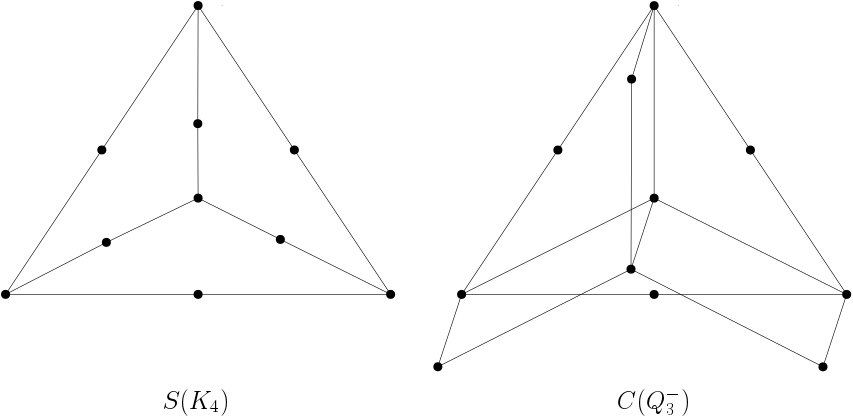}
\caption{Partial cubes $S(K_4)$  and $C(Q_3^-)$.}
\label{fig3}
\end{center}
\end{figure}

With the help of a computer we have checked that $G_8, S(K_4)$  and $C(Q_3^-)$ are the only partial cubes of isometric dimension 4 that are not $\Theta$-graceful. As it is easy to check (by hand) that all partial cubes of isometric dimension at most 3 are $\Theta$-graceful, hence $G_8$ is also the smallest example of a partial cube that is not $\Theta$-graceful.

As $G_8$ is a median graph the answer to the original question is negative also if reduced to the class of median graphs. It would be interesting to characterize which partial cubes are $\Theta$-graceful.


\begin{thebibliography}{99}

\bibitem{brkl-06}
B. Bre\v{s}ar and S. Klav\v{z}ar, $\Theta$-graceful labelings of partial cubes,
Discrete Math. 306 (2006) 1264--1271.

\bibitem{chepoi1988}
  V.~D.~Chepoi,
  $d$-Convexity and isometric subgraphs of Hamming graphs,
  Cybernetics, 1 (1988) 6--9.

\bibitem{djok-73}
D.~Djokovi\'c, Distance preserving subgraphs of hypercubes, J. Combin. 
Theory Ser. B 14 (1973) 263--267.

\bibitem{gal} J. A.~Gallian, A dynamic survey of graph labeling,
Electron. J. Combin., DS6, 535 pp (version December 15, 2019).
%http://www.combinatorics.org/ojs/index.php/eljc/article/view/DS6

\bibitem{klli-03}
S.~Klav\v{z}ar and A.~Lipovec, 
Partial cubes as subdivision graphs and as generalized Petersen graphs,
Discrete Math. 263 (2003) 157--165.

\bibitem{klmu-99}
S.~Klav\v{z}ar and H. M. Mulder, 
Median graphs: characterizations, location theory and related structures,
Journal of Combinatorial Mathematics and Combinatorial Computing 30 (1999) 103--128.

\bibitem{mulder1978}
 H.~M.~Mulder,
 The structure of median graphs,
 Discrete Math. 24 (1978) 197--204.

\bibitem{mulder1980}
 H.~M.~Mulder,
 {\em The Interval Function of a Graph},
 Math. Centre Tracts 132, Mathematisch Centrum, Amsterdam, 1980.

\bibitem{ov-11} 
S. Ovchinnikov, {\em Graphs and Cubes}, Springer, 2011.

\bibitem{ro-67}
A.~Rosa, On certain valuations of the vertices of a graph, 
{\em Theory of Graphs (Inter. Symposium, Rome, July 1967)}, 
Gordon and Breach, N. Y. and Dunod Paris (1967), 349--355. 

\bibitem{wink-84}
P.~Winkler, Isometric embeddings in products of complete graphs,
Discrete Appl. Math. 7 (1984), 221--225.

\end{thebibliography}
\end{document}